\documentclass[12pt]{amsart}
\usepackage{amssymb,mathrsfs,url}
\DeclareMathOperator{\DHyp}{DH}
\DeclareMathOperator{\Code}{Code}
\DeclareMathOperator{\Decode}{Decode}
\DeclareMathOperator{\Cube}{Cube}
\DeclareMathOperator{\Cell}{Cell}
\newcommand{\Index}{{\rm Index}}
\newtheorem{fact}{Fact} 
\newtheorem{claim}{Claim} 
\newtheorem{definition}{Definition} 
\newtheorem{theorem}{Theorem} 
\newtheorem{lemma}{Lemma} 
\title{Limits, regularity and removal for finite structures}

\author{Ashwini Aroskar}
\author{James Cummings}

\begin{document}

\maketitle

\begin{abstract}
Our work builds on known results for k-uniform hypergraphs including the existence of limits, a Regularity Lemma and a Removal Lemma. Our main tool here is a theory of measures on ultraproduct spaces which establishes a correspondence between ultraproduct spaces and Euclidean spaces. First we show the existence of a limit object for convergent sequences of relational structures and as a special case, we retrieve the known limits for graphs and digraphs. Then we extend this notion to finite models of a fixed universal theory. We also state and prove a Regularity Lemma and a Removal Lemma. We will discuss connections between our work and Razborov's flag algebras as well. 
\end{abstract}

\section{Introduction}
The theory of {\em graph limits} was initiated by Lov\'{a}sz and Szegedy \cite{LS04}. The theory associates to each sequence of graphs $(G_n)$, which is increasing and convergent in an appropriate sense, a ``limit object'' in the form of a Lebesgue-measurable function called a {\em graphon}. There are close connections between the theory of graph limits and some well-known structural facts about large finite graphs, such as, the Regularity Lemma and the Removal Lemma. The monograph by Lov\'{a}sz \cite{Lovasz-book} gives a very detailed treatment of graph limits.

Elek and Szegedy \cite{ES08b} developed a theory of measures on ultraproduct spaces and investigated limits, regularity and removal in the context of $k$-uniform hypergraphs. The limit objects are significantly more complicated in this case: a graphon is a function of two variables, but the limit of a sequence of $k$-uniform hypergraphs is a function of $2^k - 2$ variables. In this paper we build on the work of Elek and Szegedy to construct and apply limit objects for sequences of  finite relational structures, and then extend our results to the setting of  finite models of a fixed universal first-order theory $T$. Razborov's work on flag algebras \cite{Raz07} also gives a notion of a limit object for sequences of models of $T$, in the form of a certain kind of $\mathbb R$-algebra homomorphism; we will discuss the connection between our analytic limit objects and Razborov's more  algebraic limits. 
  The paper is organised as follows:
\begin{itemize}
\item  In section \ref{prelim} we cover some necessary background material. The impatient reader should probably skip this section and refer  back as necessary. In subsection \ref{strux} we  discuss relational languages and structures for such languages, define various kinds of structure-preserving  maps between structures, and then define the important notion of a {\em convergent sequence of structures.}

 In subsection \ref{coding} we define a scheme for coding an arbitrary $n$-ary relation by a family of directed hypergraphs, which we will use in section \ref{withoutT}. In subsection \ref{ultrafilters} we discuss the elementary theory of ultrafilters, ultraproducts and ultrapowers. In subsection \ref{ultrabasics} we give the basic facts about
the ``Loeb measure'' \cite{Loeb} on an ultraproduct of finite sets, and finally in subsection \ref{EStheory} we describe some results by Elek and Szegedy about the Loeb measure.

\item Section \ref{withoutT} contains our main results. In subsection \ref{limsec} we show how to associate a limit object to a convergent sequence of relational structures, and prove that we can use the limit object as a ``template'' to construct a random sequence of structures which resembles (in a precise sense) the original sequence. Our proofs use some of Razborov's flag algebra machinery, and we briefly discuss the relationship between our work and the theory of flag algebras.

 In subsection \ref{hypsec} we discuss a version of the technical notion of {\em hyperpartition}, due originally to  Elek and Szegedy \cite{ES08b}. In subsection \ref{removalsection} we prove a form of Strong Removal for finite relational structures, and finally in subsection \ref{regularity} we sketch a version of the Regularity Lemma for such structures. 

\item In section \ref{withT} we extend the theory from section \ref{withoutT} to the more general setting of structures which are required to be models of a fixed universal theory $T$. In particular we prove a version of Strong Removal in this setting. 

\item The appendix has two parts. In part \ref{examples} we work out our limit theory in a very simple case (one binary relation), and then discuss the relationship between our work and the theories of digraph limits (Offner and Pikhurko \cite{Offner}) and poset limits (Janson \cite{Janson}). In part \ref{forbid} we prove a technical result needed for the results of  section \ref{withT}.
\end{itemize}

Our notation is mostly standard; here is a brief review. If $X$ and $I$ are sets then $X^I$ is the set of functions from $I$ to $X$, which we usually think of as $I$-indexed sequences. A typical element of $X^I$ is generally denoted as $\vec x = (x_i : i \in I)$, or just as $(x_i)$ when $I$ is clear from the context.  We write $X^t$ for the set of $t$-tuples from $X$. We denote the set of $t$-tuples in $X^t$ with no repeated entry as $X^{\underline t}$. Note that this is in line with one notation used for the ``descending factorial'' $n^{\underline t} = n!/(n-t)!$, which we will adopt.

 As usual $[m] = \{1, \ldots m \}$. Note that in particular $[m]^t$ is the set of $t$-tuples of elements of $[m]$.

 If $X$ is a set and $R$ is an $n$-ary relation on $X$, then we use the equivalent notations $R(\vec x)$ or ${\vec x} \in R$ as convenient.

 In an effort to improve readability, we have adopted some typographical conventions:

\begin{itemize}
\item Objects related to logic (languages, structures, sets of structures) are denoted by calligraphic letters ($\mathcal L$ for languages, $\mathcal M$ and $\mathcal N$ for structures, $\mathcal F$ for sets of structures).
\item Indexed families of Lebesgue measurable sets (which play the role of limit objects in our theory) are denoted by gothic letters (e.g. $\mathfrak E$, $\mathfrak F$).
\item  Following Razborov \cite{Raz07}, random objects are denoted by boldface symbols, in particular $\boldsymbol{\mathcal N}$ is always a random structure.
\item  Some families of partitions known as {\em hyperpartitions} (see subsection \ref{hypsec}) are denoted by variations of the fancy calligraphic letter $\mathscr{H}$.
\end{itemize} 

\section{Preliminaries}  \label{prelim} 

\subsection{Structures} \label{strux}
   We work in a general setting which covers many combinatorially interesting structures. We will fix a {\em finite relational language} $\mathcal L$, which is specified by giving finitely many relation symbols $R_1, \ldots, R_n$ together with a natural number $r_i > 0$ for each symbol $R_i$; $r_i$ is called the {\em arity of $R_i$}. An {\em $\mathcal L$-structure} $\mathcal M$ is a set $X$  together with $r_i$-ary relations  $R_i^{\mathcal M} \subseteq X^{r_i}$, for $1 \le i \le n$. We call $X$ the {\em underlying set} of the structure $\mathcal M$ and write $X = \vert \mathcal M \vert$, so that the cardinality of the underlying set for $\mathcal M$ is denoted by $\Vert {\mathcal M} \Vert$. 
 When $\mathcal M$ is a structure and $A \subseteq \vert {\mathcal M} \vert$, we can define a new structure  with underlying set $A$ by interpreting each relation symbol $R_i$
 as $R_i^{\mathcal M} \cap A^{r_i}$, obtaining a substructure ${\mathcal M} \vert_A$. 
 
 The formulae of the relational language are built up from the given relation symbols together with a special binary relation symbol $=$,  variable symbols, connectives and quantifiers in the usual way. When we are interpreting formulae of our language, the symbol $=$ will always be interpreted as equality. Given a formula $\phi$, a structure $\mathcal M$ and elements $m_1, \ldots m_s \in \vert \mathcal M \vert$, we follow standard logical usage and write ``${\mathcal M} \models \phi(m_1, \ldots, m_s)$'' for ``$\phi$ is true of the elements $m_1, \ldots, m_s$ in the structure $\mathcal M$''.

   A {\em sentence} is a formula in which every variable is in the scope of some quantifier, and which therefore has a definite truth value in every structure. A {\em theory} is a set of sentences. A {\em model} of a theory $T$ is a structure in which every sentence from $T$ is true. A formula $\phi$ is {\em universal} if $\phi$ has the form $\forall x_1 \ldots \forall x_s \; \psi$ where $\psi$ contains no quantifiers, and a {\em universal theory} is a theory consisting of universal sentences.

 We will work in the context of a countable universal theory $T$, and we will assume that $T$ has arbitrarily large finite models. By the well-known compactness theorem from first order logic, this is equivalent to the assertion that $T$ has an infinite model.  Since $T$ is universal, whenever ${\mathcal M}$ is a model of $T$ and $A \subseteq \vert {\mathcal M} \vert$ we have that the ``induced'' substructure ${\mathcal M} \vert_A$ is also a model of $T$. 

In the following sections we will first develop a version of the theory in the setting when $T = \emptyset$, that is a theory of ``pure relational structures''. We will then 
extend the theory to models of $T$ using an analysis of models of $T$ in terms of ``forbidden induced substructures'', which we describe in section \ref{withT}.

   Given structures $\mathcal M$ and $\mathcal N$:
\begin{itemize}

\item  A {\em homomorphism from $\mathcal M$ to $\mathcal N$} is a function $f: \vert {\mathcal M} \vert \rightarrow \vert {\mathcal N} \vert$ such that
\[
 R_i^{\mathcal M}(m_1, \ldots, m_{r_i}) \implies R_i^{\mathcal N}(f(m_1), \ldots, f(m_{r_i}))
\]
 for all $i$ with $1 \le i \le n$ and all $r_i$-tuples $(m_1, \ldots, m_{r_i})$ of elements of $\vert {\mathcal M} \vert$.

\item  An {\em embedding of $\mathcal M$ into $\mathcal N$} is an injective function $f: \vert {\mathcal M} \vert \rightarrow \vert {\mathcal N} \vert$ such that
\[
 R_i^{\mathcal M}(m_1, \ldots, m_{r_i}) \iff R_i^{\mathcal N}(f(m_1), \ldots, f(m_{r_i}))
\]
 for all $i$ with $1 \le i \le n$ and all $r_i$-tuples $(m_1, \ldots, m_{r_i})$ of elements of $\vert {\mathcal M} \vert$.

\item   An {\em isomorphism between $\mathcal M$ and $\mathcal N$} is a bijective embedding of ${\mathcal M}$ into ${\mathcal N}$. The structures ${\mathcal M}$ and ${\mathcal N}$ are {\em isomorphic} if and only if there is an isomorphism between them, and in this case we write ${\mathcal M} \simeq {\mathcal N}$. 

\end{itemize} 

 We now define various quantities which measure (in slightly different senses) the ``density of ${\mathcal M}$ in ${\mathcal N}$'' for finite $\mathcal L$-structures $\mathcal M$ and $\mathcal N$. The important quantities for us are $p$ and $t_{\rm ind}$; we include the quantities $t$ and $t_0$, which play more of a role in the classical theory
 of graph and hypergraph limits, for the sake of completeness. 
\begin{definition}  
Let ${\mathcal M}$ and ${\mathcal N}$ be finite $\mathcal L$-structures.
\begin{enumerate}
\item  The {\em induced substructure density} $p({\mathcal M}, {\mathcal N})$ is the probability that ${\mathcal N} \vert_{\boldsymbol A} \simeq M$, where $\boldsymbol A$ is a subset of $\vert {\mathcal N} \vert$ of cardinality $\Vert {\mathcal M} \Vert$ chosen uniformly at random. 
\item  The {\em homomorphism density} $t({\mathcal M}, {\mathcal N})$ is the probability that a function ${\boldsymbol f}: \vert {\mathcal M} \vert \rightarrow \vert {\mathcal N} \vert$ chosen uniformly at random is a homomorphism.
\item The {\em injective homomorphism density} $t_0({\mathcal M}, {\mathcal N})$ is the probability that an injective function ${\boldsymbol f}: \vert {\mathcal M} \vert \rightarrow \vert {\mathcal N} \vert$ chosen uniformly at random is a homomorphism.
\item The {\em embedding density} or {\em induced homomorphism density}  $t_{\rm ind}({\mathcal M}, {\mathcal N})$ is the probability that an injective  function ${\boldsymbol f}: \vert {\mathcal M} \vert \rightarrow \vert {\mathcal N} \vert$ chosen uniformly at random is an embedding.
\end{enumerate}
\end{definition} 

 By convention the quantities $p({\mathcal M}, {\mathcal N})$, $t_0({\mathcal M}, {\mathcal N})$ and $t_{\rm ind}({\mathcal M}, {\mathcal N})$ are zero when $\Vert {\mathcal N} \Vert <  \Vert {\mathcal M} \Vert$. 
\begin{definition}
 A sequence $( {\mathcal N}_k )$ of structures is {\em increasing } if and only if $\Vert {\mathcal N_k} \Vert \rightarrow \infty$, and {\em convergent} if and only if it is increasing and additionally the sequence of induced substructure densities $(p({\mathcal M}, {\mathcal N}_k))$ converges for every finite structure ${\mathcal M}$. 
\end{definition} 
 Since there are only countably many finite ${\mathcal L}$-structures, an easy diagonal argument along the lines of the Bolzano-Weierstrass theorem shows that every increasing sequence has a convergent subsequence. 

   It is easy to see that if $( {\mathcal N}_k )$  is convergent then each of the sequences $(t({\mathcal M}, {\mathcal N}_k))$,   $(t_0({\mathcal M}, {\mathcal N}_k))$,  and  $(t_{\rm ind}({\mathcal M}, {\mathcal N}_k))$ converges for all $\mathcal M$, and their limiting values may be computed from the limiting values for $(p({\mathcal M}, {\mathcal N}_k))$. Given a convergent sequence $( {\mathcal N}_k )$, we will define a function $\Phi_p$ by setting $\Phi_p({\mathcal M}) = \lim_{k \rightarrow \infty} p( {\mathcal M}, {\mathcal N}_k)$: of course the definition depends on the convergent sequence $( {\mathcal N}_k )$, but this should always be clear from the context.
 The functions $\Phi_t$, $\Phi_{t_0}$  and $\Phi_{t_{\rm ind}}$ are defined similarly. 
 It is easy to see that $\vert t({\mathcal M}, {\mathcal N}) - t_0({\mathcal M}, {\mathcal N}) \vert$ is $O ( \Vert \mathcal N \Vert^{-1} )$, so that $\Phi_t = \Phi_{t_0}$.

\begin{fact} \label{ptfact} Let ${\mathcal M}$ and $\mathcal N$ be finite ${\mathcal L}$-structures.
 Then  $t_{\rm ind} ({\mathcal M}, {\mathcal N}) = p({\mathcal M}, {\mathcal N}) t_{\rm ind}({\mathcal M}, {\mathcal M})$.
\end{fact}
\begin{proof}
 Consider the process of choosing a random injection $\boldsymbol f$ from $\vert {\mathcal M} \vert$ to $\vert {\mathcal N} \vert$ as a two-step process, in which we first choose the range of $\boldsymbol f$ and then choose its values. For this process to yield an embedding, the first step must yield a set inducing a  substructure isomorphic to $\mathcal M$, and the second step must yield an isomorphism between $\mathcal M$ and this isomorphic copy of $\mathcal M$.
\end{proof}
  
 It follows that $\Phi_{t_{\rm ind}}( {\mathcal M} ) =  \Phi_p( {\mathcal M}) t_{\rm ind}({\mathcal M},{\mathcal M})$. The quantity $t_{\rm ind}({\mathcal M}, {\mathcal M})$ is the probability that a random permutation of $\vert {\mathcal M} \vert$ is an automorphism of $\mathcal M$.  

\subsection{Decomposing $t$-ary relations into uniform directed hypergraphs} \label{coding}
    For technical reasons we will find it helpful to break up  $t$-ary relations into simpler pieces. To see why this is so, consider the language with one binary relation $R$ and let ${\mathcal M}_n$ be a structure with $n$ elements such that $R(x, y) \iff x = y$. Clearly $( {\mathcal M}_n )$  is a convergent sequence, and the probability that two randomly chosen elements are related tends to zero as $n \rightarrow \infty$; but this ``two dimensional'' information is not enough to determine the limiting behaviour of $(p({\mathcal M}, {\mathcal M}_n))$ for all finite structures $\mathcal M$. We also need to know the ``one dimensional information'' that the probability that a  single randomly chosen element is related to itself tends to one.

    Recall that an {\em $r$-uniform directed hypergraph} is a pair $(X, E)$ where $X$ is a set and $E \subseteq X^{\underline r}$. Given a $t$-ary relation $R$ on $X$, we will associate to $R$ a family of directed hypergraphs which code it. Let ${\rm Part}_t$ be the set of partitions of $[t]$, and for each partition $p \in {\rm Part}_t$ let $\Vert p \Vert$ be the number of classes in $p$. If $\Vert p \Vert = r$ then we will enumerate the $r$ classes of $p$ as $C^p_1, \ldots, C^p_r$ where $\min(C^p_i)$ increases with $i$.

  Given a $t$-tuple $\vec x \in X^t$, we let $p(\vec x) \in {\rm Part}_{[t]}$ be the partition in which $j$ and $j'$ lie in the same class if and only if $x_j = x_{j'}$. If $\Vert p(\vec x) \Vert = r$, then we let $C(\vec x)$ be the $r$-tuple obtained from $\vec x$ by deleting elements from $\vec x$ as follows: we retain the first occurrence of each element of $\{ x_i : 1 \le i \le t \}$, but delete all subsequent occurrences. Given a $t$-ary relation $R \subseteq X^t$, we let $R_p = \{ \vec x \in R : p(\vec x) = p \}$ and let $\DHyp^R_p = \{ C(\vec x) : \vec x \in R_p \}$ for each $p \in {\rm Part}_t$. For example if $R$ is a $3$-ary relation then $R$ will be coded by the directed hypergraphs with edge sets:

\begin{itemize}

\item $\DHyp^R_{ \{ \{ 1, 2, 3 \} \} }  = \{ x \in X : R(x, x, x) \}$.

\item $\DHyp^R_{ \{ \{ 1, 2 \}, \{ 3 \} \} }  = \{ (x, y) \in X^{\underline 2}: R(x, x, y) \}$.

\item $\DHyp^R_{ \{ \{ 1, 3 \}, \{ 2 \} \} } = \{ (x, y) \in X^{\underline 2}: R(x, y, x) \}$. 

\item $\DHyp^R_{ \{ \{ 1 \}, \{ 2,  3 \} \} } = \{ (x, y) \in X^{\underline 2}: R(x, y, y) \}$.

\item $\DHyp^R_{ \{ \{ 1 \}, \{  2 \}, \{ 3 \} \} } = \{ (x, y, z) \in X^{\underline 3}: R(x, y, z) \}$. 

\end{itemize} 
  
  When $\mathcal N$ is an $\mathcal L$-structure, we write $\DHyp^{ {\mathcal N}, i}_p$ for $\DHyp^{ {R_i}^{\mathcal N} }_p$. The structure  ${\mathcal N}$ with underlying set $X$ can be coded by the family $\Code({\mathcal M}) = (\DHyp^{ {\mathcal N}, i}_p : i \in [n], p \in {\rm Part}_{[r_i]})$. For notational convenience later we define $\Index$ to be the set of pairs $(i, p)$ with $i \in [n]$ and $p \in {\rm Part}_{[r_i]}$, which forms the index set for $\Code({\mathcal M})$. 
  If $\DHyp = (\DHyp^i_p : (i, p) \in \Index )$ is a family of directed hypergraphs of the appropriate types on some set $X$, then we write $\Decode(\DHyp)$ for the unique structure on $X$ coded by $\DHyp$.    
  
\subsection{Ultrafilters and ultraproducts} \label{ultrafilters}
   A {\em filter on $\mathbb N$}  is a family $F$ of subsets of $\mathbb N$ such that $\emptyset \notin F$, ${\mathbb N} \in F$, $F$ is upwards closed (in the sense that $A \in F$ and $A \subseteq B \subseteq {\mathbb N}$ implies $B \in F$), and finally $F$ is closed under finite intersections. An {\em ultrafilter on $\mathbb N$} is a filter $U$ that is maximal under inclusion, equivalently a filter such that for every set $A$, exactly one of the sets $A$ and ${\mathbb N} \setminus A$ lies in $U$.  It is easy to see that $U$ is an ultrafilter if and only if $U$ is of the form $\{ A : \mu(A) = 1 \}$, for some finitely additive probability measure $\mu$ that is defined on all subsets of $\mathbb N$ and  takes only the values $0$ and $1$.  If an ultrafilter $U$ contains a finite set then it has the form $\{ A : n \in A \}$ for some $n$ and is said to be {\em principal}: an easy application of Zorn's lemma shows that any infinite set lies in a non-principal ultrafilter.
   For the rest of this paper we assume that we have fixed a non-principal ultrafilter $U$ on $\mathbb N$. Given a sequence of $\mathcal L$-structures $( {\mathcal M}_k )_{k \in {\mathbb N}}$, a standard construction in logic is the formation of a new $\mathcal L$-structure called the {\em ultraproduct}.
 The standard notation for this structure is $\prod_k  {\mathcal M}_k/U$; in this discussion we will call it ${\mathcal M}_\infty$. In the special case when ${\mathcal M}_k = {\mathcal M}$ for all $k$, we write ${\mathcal M}^{\mathbb N}/U$ for the corresponding ultraproduct which is called the {\em ultrapower of $\mathcal M$ by $U$.} 
 To define ${\mathcal M}_\infty$  we consider sequences $(x_k) \in \prod_k \vert {\mathcal M}_k \vert$, introduce an equivalence relation 
\[   
 (x_k) =_U (y_k) \iff \{ k : x_k = y_k \} \in U,
\]
 and let the underlying set $\vert {\mathcal M}_\infty \vert$ be the set of equivalence classes. We denote the equivalence class of $(x_k)$ by $[x_k]$ and interpret the language
 as follows: for each relation symbol $R$, if $R$ has arity $r$ then
\[
  R^{{\mathcal M}_\infty}([x^1_k], \ldots, [x^r_k]) \iff \{ k : R^{{\mathcal M}_k}(x^1_k, \ldots, x^r_k) \} \in U.
\]
  The crucial fact about ultraproducts is  {\L}o{\'s}'  Theorem: for {\em any} first-order formula $\phi$,
\[
  {\mathcal M}_\infty \models \phi([x^1_k], \ldots, [x^r_k]) \iff \{ k : {\mathcal M}_k \models \phi(x^1_k, \ldots, x^r_k) \} \in U.
\]
This is proved by a straightforward induction on the structure of the formula $\phi$, with the interpretation of the relation symbols as the base case.  We note the easy corollary that for any theory $T$, the class of models of $T$ is closed under ultraproducts.       
   For our purposes it will be very convenient to use ideas from set theory. We recall that the {\em language of set theory} has one binary relation symbol $\in$ denoting membership. ZFC set theory is a first-order theory written in this language, whose intended model is the class $V$ of all sets; we can view almost all mathematical objects as sets, and can justify almost all mathematical constructions and proofs on the basis of ZFC. We will form the ultrapower \footnote{Formally speaking we should be careful here, because each equivalence class (mod $U$) of functions is a proper class and it would be improper to form the collection of equivalence classes. This issue can be dealt with by an argument due to Dana Scott \cite{Scott} in which each equivalence class is replaced by the set of elements within it of minimal rank.} $V^* = V^{\mathbb N}/U$.

 By the preceding discussion, $V^*$ is itself a model of ZFC  and contains versions of almost all standard mathematical objects. If $\in^*$ is $V^*$'s version of the membership relation, then we may view each element $b$ of $V^*$ as a set in its own right whose elements are those $a  \in V^*$ such that $V^* \models a \in^* b$; with this in mind we will refer to such objects $a$ as the {\em elements of $b$} and write ``$a \in b$'' rather than ``$V^* \models a \in^* b$''. In particular we may view the ultraproduct ${\mathcal M}_\infty$ as $[{\mathcal M}_i]$,  a structure lying in $V^*$.  

 Given any set $x \in V$, we let $c_x$ be the sequence which is constant with value $x$ and then $x^*$ be the element of $V^*$ represented by $c_x$; it follows immediately from the definitions and {\L}o{\'s}' Theorem that $x \mapsto x^*$ is an injective map and
\[
 V \models \phi(x_1, \ldots, x_n) \iff V^* \models \phi(x_1^*, \ldots, x_n^*)
\]
 for all $\phi$. In the terminology of logic we have an ``elementary embedding from $V$ to $V^*$''.

 If $X$ is an infinite set then $X^*$ contains elements which are not of the form $x^*$ for any $x \in X$. What is more, if $Y$ is the power set of $X$ then by elementarity, $Y^*$ is the set of subsets of $X^*$ which lie in $V^*$; when $X$ is infinite it turns out that $Y^*$ has elements which are not of the form $A^*$ for any $A \in Y$, and in addition there exist subsets of $X^*$ which are not in $Y^*$.
   To clarify these points, consider the example $X = {\mathbb N}$. It is not hard to see that $\{ n^* : n \in {\mathbb N} \}$ is an initial segment of ${\mathbb N}^*$ which is isomorphic to $\mathbb N$, and that setting $x_k = k$ we obtain an element $[x_k]$ of ${\mathbb N}^*$ which is greater than $n^*$ for all $n$. If we now let $y_k = \{ 0, 1, \ldots, k \}$ then by {\L}o{\'s}' theorem $V^* \models \mbox{``$[y_k]$ is finite''}$, and it follows that $[y_k]$ is a subset of ${\mathbb N}^*$ which lies in $V^*$ but is not of the form $A^*$ for any $A \subseteq {\mathbb N}$. 
 Additionally we claim that the set $\{ n^* : n \in {\mathbb N} \}$ is not in $V^*$; for if it were, then by elementarity its complement would have a minimal element $y$ say,  but by elementarity again $y$ would have a predecessor $n^*$ leading to an immediate contradiction.

 The structure $V^* = V^{\mathbb N}/U$ has a compactness property known as {\em $\aleph_1$-saturation}. This states that if $\{ \phi_i(x) : i \in {\mathbb N} \}$ is a set of formulae, and for every $j$ there is $a_j \in V^*$ such that $\phi_i(a_j)$ holds in $V^*$ for all $i \le j$, then there is $a \in V^*$ such that $\phi_i(a)$ holds in $V^*$ for 
 all $i$. The argument is a routine diagonalisation using {\L}o{\'s}' theorem.

 Let $X$ be a compact Hausdorff space and let $y \in X^*$. A routine argument shows that there is a unique $x \in X$ such that $y \in O^*$ for every open $O \ni x$; we call this
point $x$ the {\em standard part} of $y$ and write $x = std(y)$. A more topological view of this construction can be obtained by observing that if $y = [y_k]$ then $x$ is the unique point such that $\lim_{k \rightarrow U} y_k = x$, that is  for all open $O \ni x$ we have $\{ k : y_k \in O \} \in U$.

\subsection{Measure theory on ultraproducts: basic facts} \label{ultrabasics}
   Measures on ultraproducts of finite sets are a special case of {\em Loeb measures} \cite{Loeb}, which are well-known objects in non-standard analysis. We review some basic facts.    
   Let $(X_k)$ be a sequence of finite sets with $\vert X_k \vert \rightarrow \infty$, and let $X_\infty = [X_k]$, so that 
\[
  V^* \models \mbox{``$X_\infty$ is finite''}.
\]
 In actuality $X_\infty$ is easily seen to be an uncountable set. 
 Our goal is to define a measure on $X_\infty$. We will start by assigning a measure to those subsets of $X_\infty$ which lie in $V^*$.  
   Let $A_\infty \in V^*$ be a subset of $X_\infty$. We begin by working in $V^*$ and computing the measure of $A_\infty$ with respect to the ``normalised counting measure'' on $X_\infty$, that is $\frac{\vert A _\infty \vert}{\vert X_\infty \vert} \in [0,1]^*$. Since $[0, 1]$ is compact this number has a standard part which we call $\mu(A_\infty)$. 
 For a more concrete definition, we may choose $A_k \subseteq X_k$ such that $A_\infty = [A_k]$ and then check that $\mu(A_\infty) = \lim_{k \rightarrow U} \frac{\vert A_k \vert}{\vert X_k \vert}$.   
 
 The subsets of $X_\infty$ which lie in $V^*$ form a Boolean algebra. To extend $\mu$ to a wider class of sets, we say that an arbitrary set $A \subseteq X_\infty$ is {\em null} if and only if for every $\epsilon > 0$ there is $B \in V^*$ such that $A \subseteq B \subseteq X_\infty$ and $\mu(B) < \epsilon$. We let $\sigma$ be the class of subsets $A \subseteq X_\infty$ such that $A \Delta B$ is null for some $B \in V^*$, and extend $\mu$ onto $\sigma$ by defining $\mu(A) = \mu(B)$ for some (any) $B \in V^*$ such that $A \Delta B$ is null.
 The following fact is a consequence of the $\aleph_1$-saturation of $V^*$. 
\begin{fact} 
$\sigma$ is a $\sigma$-algebra and $\mu$ is a $\sigma$-additive probability measure. 
\end{fact}

\subsection{Measure theory on ultraproducts: results of Elek and Szegedy} \label{EStheory}
   Elek and Szegedy \cite{ES08b} proved a family of theorems about measure theory on ultraproducts which will play a central role in our results. We give a brief account of the results we need. 

Let $A \subseteq [t]$ with $\vert A \vert = s$, let $A$ be enumerated in increasing order as $i_1, \ldots i_s$, and  let $\pi^t_A$ denote the map on the class of $t$-tuples which returns for each $t$-tuple $(a_i :i \in [t])$ the $s$-tuple $(a_{i_j} : j \in [s])$.  Then $\sigma^t_A = \{ {\pi^t_A}^{-1}[B]: B \in \sigma_{[s]} \}$, that is to say elements of $\sigma^t_A$ are measurable sets which depend only on coordinates in $A$. It is easy to see that $\pi^t_A$ sets up an isomorphism between $(X_\infty^t, \sigma^t_A, \mu_{[t]} \restriction \sigma^t_A)$ and $(X_\infty^s, \sigma_{[s]}, \mu_{[s]})$.  To lighten the notation, we will often drop the superscript $t$ when it is clear from the context.   
  The following result is a version of a well-known theorem of Maharam \cite{Maharam}.

\begin{fact}   Let $(X, \tau, \nu)$ be a separable atomless complete measure space with $\nu(X) =1$.
  Then the associated measure algebra is isomorphic to the measure algebra of the unit interval with Lebesgue measure. Moreover there is a map $f: X \rightarrow [0,1]$ defining an isomorphism between these measure algebras, that is:

\begin{itemize}

\item  For every measurable $B \subseteq [0,1]$, $f^{-1}[B] \in \tau$ and $\lambda(B) = \mu(f^{-1}[B])$.

\item  For every $A \in \tau$ there is a measurable set $B \subseteq [0,1]$ such that $A$ and $f^{-1}[B]$ differ by a null set.

\end{itemize} 

\end{fact}

For a non-empty set $A \subseteq [t]$, we define $\sigma_A^*$ to be the subalgebra of $\sigma_A$ generated by the union of the algebras $\sigma_B$ for $B \subseteq A$, $\vert B \vert = \vert A \vert -1$. It is a surprising fact that $\sigma_A$ is much larger than $\sigma_A^*$; this plays a crucial role in everything that follows.

   Let $A \subseteq [t]$ be non-empty and let $0 < s \le \vert A \vert$, then we define $r(A, s)$ to be the set of subsets $B \subseteq A$ with $0 < \vert B \vert \le s$. As a special case $r(A) = r(A, \vert A \vert)$, that is the set of non-empty subsets of $A$. The group $S_{[t]}$ of permutations of $[t]$ acts in the natural way on subsets of $[t]$  by permuting their entries, and on $t$-tuples by permuting their coordinates, and we write $z^\sigma$ for the image of $z$ by $\sigma$; we get an induced action on sets $Z$ of $t$-tuples by defining $Z^\sigma = \{ z^\sigma : z \in Z \}$, and so on inductively for sets of sets of $t$-tuples etcetera.
  When $A \subseteq [t]$ we write $S_A$ for the group of permutations in $S_{[t]}$ which fix points outside $A$; we will often identify this with the permutation group of $A$ in the obvious way. 
\begin{definition} A {\em separable system} on $X_\infty^t$ is a family $\{ l_A: A \in r([t]) \}$ such that

\begin{enumerate}

\item $l_A$ is an atomless separable subalgebra of $\sigma_A$, which is independent of $\sigma_A^*$ (that is to say sets in $l_A$ are independent of $\sigma_A^*$).

\item $(l_A)^\pi = l_{A^\pi}$ for all $A$.

\item $Y^\pi = Y$ for all $Y \in l_A$ and $\pi \in S_A$. 

\end{enumerate}

\end{definition} 
   To illustrate the symmetry and independence properties consider the case when $t =2$, enumerating $r([2])$ as $\{ 1 \}$, $\{ 2 \}$, $\{ 1, 2 \}$; the separable system consists of algebras $l_{\{1\}}, l_{\{2\}}, l_{\{1,2\}}$ where exchanging coordinates $1$ and $2$ exchanges the algebras $l_{\{1\}}$ and  $l_{\{2\}}$, while every set in $l_{\{1,2\}}$ is symmetric under this exchange and is independent of rectangles.

   It is a key fact that arbitrary measurable sets can be well-approximated by separable systems:
\begin{fact} \label{systemfact} Let  $\langle X^B_n : n \in {\mathbb N}, B \in r([t]) \rangle$ be such that $X^B_n \in \sigma_B$ for all $n$ and $B$. Then there exists a separable system $\{ l_A: A \in r([t]) \}$ such that, for every $n$, $X^B_n$ differs on a set of measure zero from a set in the $\sigma$-algebra generated by $\bigcup_{A \subseteq B} l_A$. 
\end{fact}
  In this situation we will say that the separable system is {\em suitable} for the sets $X^B_n$. 

    Given a separable system on $X_\infty^t$, we can ``realise'' it in a symmetry-preserving way on the space $[0,1]^{r([t])}$ (considered as a measure space equipped with Lebesgue measure). 

\begin{definition} \label{realisationfact} If $\{ l_A: A \in r([t]) \}$ is a separable system, then a {\em separable realisation} of the system is a measure-preserving map $\phi : X_\infty^t \rightarrow [0,1]^{r([t])}$ such that (writing $\phi_A(x)$ for the $A$-component of $\phi(x)$)

\begin{enumerate}

\item  $\phi_A$ only depends on coordinates in $A$.

\item For every measurable set $B \subseteq [0,1]$, $\phi_A^{-1}(B) \in l_A$.

\item For every permutation $\sigma \in S_{[t]}$, $\phi_{A^\sigma}(x^\sigma) = \phi_A(x)$.    

\end{enumerate}

\end{definition}
   It follows easily from the properties of such a map $\phi$ that there exist maps $\phi^j : X_\infty^j \rightarrow [0,1]$ for $j \in [t]$ such that:

\begin{enumerate} 

\item $\phi^j$ is invariant under permutation of its arguments.

\item For all $A = \{ a_1, \ldots, a_n \} \in r( [t] )$, $\phi_A(x_1, \ldots, x_t) = \phi^n(x_{a_1}, \ldots, x_{a_t})$.

\end{enumerate} 
 We will refer to the maps $\phi^j$ as the {\em levels} of the realisation map $\phi$. For example when $t =2$, we may write $\phi(x_1, x_2)  = (\phi^1(x_1), \phi^1(x_2), \phi^2(x_1, x_2) )$ where $\phi^2$ is a symmetric function. 
   Since we are interested in structures of varying sizes with relations of varying arities, we need to vary the arities of the functions in a realisation.

 \begin{fact} \label{liftingfact} There exists $\phi : X_\infty^t \rightarrow [0,1]^{r([t])}$, a separable realisation of $\{ l_A : A \in r([t]) \}$  with  levels $\phi^j$ for $j \in [t]$.
\begin{itemize}
\item (Restriction) If $u \le t$, then $\phi$ induces a realisation $\psi: X_\infty^{u} \rightarrow [0,1]^{r([u])}$ defined by setting $\psi_{\{ b_1, \ldots, b_m\}} (x_1, \ldots, x_{u}) = \phi^m(x_{b_1}, \ldots, x_{b_m})$.

\item (Lifting) If $u \ge t$, then $\phi$ induces a measure preserving map $\psi: X_\infty^u \rightarrow [0,1]^{ r([u], t) }$ defined by setting $\psi_{\{ b_1, \ldots, b_m \}} (x_1, \ldots x_{u}) = \phi^m(x_{b_1}, \ldots, x_{b_m})$.
\end{itemize}
\end{fact} 
   In a mild abuse of notation we will not distinguish carefully between the original realisation $\phi$ and its various restrictions and liftings. 

\section{Limits, regularity and removal for pure relational structures} \label{withoutT}
   In this section we develop some technical machinery and then prove versions of the existence of limits, Strong Removal and Regularity for finite relational structures. Many of our arguments run parallel to those given by Elek and Szegedy \cite{ES08b} for the case of uniform hypergraphs.

 In the first part of this section we define limit objects for convergent sequences $(\mathcal N_k)$ of finite structures. We proceed as follows:
\begin{itemize}

\item  Form the ultraproduct ${\mathcal N}_\infty$.

\item  Use separable realisations to associate to ${\mathcal N}_\infty$ a family ${\mathfrak E} =  (E^i_p : (i, p) \in \Index )$ of Lebesgue measurable sets where $E^i_p \subseteq [0,1]^{r[\Vert p \vert]}$.

\item  Associate to $\mathfrak E$ a random sequence $(\boldsymbol{\mathcal N}_k)$  of finite structures, and prove that almost surely $(\boldsymbol{\mathcal N}_k)$  is convergent and  $\lim_{k \rightarrow \infty} p( {\mathcal M}, \boldsymbol{\mathcal N}_k) = \lim_{k \rightarrow \infty} p( {\mathcal M}, {\mathcal N}_k)$ for all finite structures $\mathcal M$. 

\end{itemize} 
   In the remainder of the section we use the existence of limits to prove general versions of Strong Removal and Regularity.

\subsection{Limits} \label{limsec} 
    We are given a convergent sequence $( {\mathcal N}_k )$. Recall that we defined a function $\Phi_p$ on the class of finite ${\mathcal L}$-structures by setting $\Phi_p({\mathcal M}) = \lim_{k \rightarrow \infty} p( {\mathcal M}, {\mathcal N}_k)$. 
    
  We will begin by forming the ultraproduct ${\mathcal N}_\infty = \prod_k  {\mathcal N}_k/U$. Let $X_k = \vert {\mathcal N}_k \vert$ and $X_\infty = \prod_k  X_k/U$, so that ${\mathcal N}_\infty$ is an ${\mathcal L}$-structure with underlying set $X_\infty$. The first key point is that the structure ${\mathcal N}_\infty$ captures the function $\Phi_p$.
\begin{claim} \label{claimone} For every finite ${\mathcal L}$-structure $\mathcal M$, $\Phi_p({\mathcal M})$ is the standard part of $p({\mathcal M}^*, {\mathcal N}_\infty)$,
and similar statements hold for the density functions $t$, $t_0$ and $t_{\rm ind}$. 

\end{claim}
\begin{proof}   
 The standard part of $p({\mathcal M}^*, {\mathcal N}_\infty)$ is $\lim_{k \rightarrow U } p({\mathcal M}, {\mathcal N}_k)$, that is the unique real number $r$ such that for every $\epsilon > 0$, $\{ k : p({\mathcal M}, {\mathcal N}_k) \in (r - \epsilon, r + \epsilon) \} \in U$.
 Since $p({\mathcal M}, {\mathcal N}_k) \rightarrow \Phi_p(M)$ we see that $r = \Phi_p(M)$. The proofs for the other density functions are exactly similar.
\end{proof} 
  
    We now recall the coding scheme from subsection \ref{coding}, in which an arbitrary $\mathcal L$-structure $\mathcal N$ is coded by a family $\Code({\mathcal N}) = (\DHyp_p^{ {\mathcal N}, i}: (i, p) \in \Index)$.  Since the coding scheme is defined in a  uniform way, it is easy to see that the family of hypergraphs $\Code( {\mathcal N}_\infty )$ is represented in the ultrapower by the sequence $(\Code( {\mathcal N}_k) : k \in {\mathbb N})$. 
  It is immediate from the definitions that if $p \in {\rm Part}_{[r_i]}$ is a partition with $\Vert p \Vert = t$, then the relation $\DHyp_p^{ {\mathcal N}_\infty, i}$ is represented by $(\DHyp_p^{ {\mathcal N}_k, i} : k \in {\mathbb N})$, in particular it is a subset of $X_\infty^{t}$ which lies in the $\sigma$-algebra $\sigma_{[t]}$.

 Let $r_{\rm max} = \max_i r_i$.
 Appealing to Facts \ref{systemfact} and \ref{liftingfact} we may find a separable system $\{ l_A : A \in r([r_{\rm max}]) \}$ which is suitable for all the sets appearing in $\Code( {\mathcal N}_\infty )$, and then a realisation $\phi: X_\infty^{r_{\rm max}} \rightarrow [0, 1]^{r([r_{\rm max}])}$ for this system. In line with the discussion from section \ref{EStheory}, for all $t \le r_{\rm max}$ we obtain restricted separable realisations $\phi: X_\infty^t \rightarrow [0, 1]^{r([t])}$. 
 
 We associate a measurable set in a suitable power of the unit interval to each set of the form $\DHyp_p^{ {\mathcal N}_\infty, i}$. Let $\Vert p \Vert = t$, and choose $E_p^i \subseteq [0,1]^{r[t]}$ such that $\DHyp_p^{ {\mathcal N}_\infty, i} \Delta \phi^{-1} [E_p^i]$ is a null set for the measure $\mu_{[t]}$.
 We note that the definition of $E_p^i$ involves using restrictions of the original map $\phi$, and is legitimate because $\Vert p \Vert \le r_i \le r_{\rm max}$ for all $(i, p) \in \Index$.

 Let ${\mathfrak E} = (E_p^i: (i, p) \in \Index)$; systems of measurable sets of this form will serve as limit objects for convergent sequences of ${\mathcal L}$-structures. 

\begin{definition}  An {\em ${\mathcal L}$-limit} is a family of sets ${\mathfrak F} =  (F_p^i: (i, p) \in \Index)$ such that  $F_p^i$ is a measurable subset of $[0,1]^{r[\Vert p \Vert]}$  for each $p$ and $i$.

\end{definition} 
  
  Let $\mathfrak F$ be an ${\mathcal L}$-limit. We will view $\mathfrak F$  as a template for the construction, for each integer $m$, of a random $\mathcal L$-structure $\boldsymbol{\mathcal N}( {\mathfrak F}, m)$ with underlying set $[m]$.
To define this random structure we will choose uniformly at random a tuple of real numbers $\vec y \in [0,1]^{r([m], r_{\rm max})}$, and compute from this tuple a structure
 ${\mathcal N}( {\mathfrak F}, m, \vec y)$ with underlying set $[m]$.

 To define the structure ${\mathcal N}( {\mathfrak F}, m, \vec y)$, we describe its coding by directed hypergraphs. For each  $(i, p) \in \Index$ and each $t$-tuple  $(b_1, \ldots, b_t) \in [m]^{\underline t}$, let $\DHyp_p^{ {\mathcal N}({\mathfrak F}, m, {\vec y}), i}(b_1, \ldots, b_t)$ if and only if  $\vec z \in F_p^i$,
where  $\vec z = (z_B : B \in r([t]))$ is given by  $z_B = y_{ \{b_l : l \in B \} }$.

For the record we also give an uncoded version. Let  $\vec a = (a_1, \ldots, a_{r_i})$ be an $r$-tuple of elements of $[m]$. Let $p = p(\vec a) \in {\rm Part}_{[r_i]}$ be the induced partition of $[r_i]$, in which $j$ and $j'$ are in the same class if and only if $a_j = a_{j'}$, and let $t = \Vert p \Vert$.

 Let $\vec b = C(\vec a)$ be the $t$-tuple obtained from $\vec a$ by deleting repetitions, so that $\vec b \in [m]^{\underline t}$, and let $\vec z = (z_B : B \in r([t]))$ be given by  $z_B = x_{ \{b_l : l \in B \} }$.
 Then $R_i^{ {\mathcal N}({\mathfrak F}, m, {\vec y}), i}(a_1, \ldots, a_{r_i})$ if and only if  $\vec z \in F_p^i$.

We will now form the random sequence $(\boldsymbol{\mathcal N}({\mathfrak E}, j^2) : j \in {\mathbb N})$, which  (formally speaking) is a random variable defined on the product over $j \in {\mathbb N}$ of the probability spaces associated with the random structures $\boldsymbol{\mathcal N}({\mathfrak E}, j^2)$.
 We  will prove that almost surely the random sequence $(\boldsymbol{\mathcal N}({\mathfrak E}, j^2))$ is convergent and resembles the original sequence $( {\mathcal N}_k )$ in a certain precise sense.
   Our proof uses ideas from work by Razborov \cite{Raz07}. Razborov associates to each convergent sequence of finite ${\mathcal L}$-structures\footnote{Actually Razborov's theory is more general, and applies to the class of models of a universal theory containing a labelled copy of a given finite structure.}  a limit object of a rather different kind, namely a homomorphism from a certain ${\mathbb R}$-algebra (which depends only on ${\mathcal L}$) to  $\mathbb R$. The following fact appears (in a more general form) as Theorem 3.18 in \cite{Raz07}. 
\begin{fact} \label{Razfact}
Let $( {\mathcal N}_k : k \in {\mathbb N})$ be a convergent sequence of finite $\mathcal L$-structures.
 For each $j$, make the set of isomorphism classes of structures ${\mathcal M}$ with $\Vert {\mathcal M} \Vert = j^2$  into a finite probability space in which each $\mathcal M$ is assigned probability $\Phi_p({\mathcal M})$. Let $(\boldsymbol {\mathcal N}_j^* : j \in {\mathbb N})$ be the associated random sequence of structures (that is it's the random variable on the product of the finite probability spaces given by the identity function on this product). Then almost surely $(\boldsymbol {\mathcal N}_j^* : j \in {\mathbb N})$ is convergent and $\lim_{j \rightarrow \infty} p({\mathcal M}, \boldsymbol{\mathcal N}_j^*) = \Phi_p({\mathcal M})$ for all finite ${\mathcal L}$-structures $\mathcal M$. 
\end{fact}   

\begin{theorem} \label{limittheorem}  Almost surely  $(\boldsymbol{\mathcal N}({\mathfrak E}, j^2))$ is a convergent sequence and 
\[  
  \lim_{j \rightarrow \infty} p({\mathcal M}, \boldsymbol{\mathcal N}({\mathfrak E}, j^2)) = \Phi_p({\mathcal M})
\]
for all finite structures $\mathcal M$. 
\end{theorem} 
    Before starting the proof, we define some auxiliary sets relating to the density function $t_{\rm ind}$.

\begin{definition} 

 Let $\mathcal M$ be a finite ${\mathcal L}$-structure with underlying set $X$, enumerated as $x_1, \ldots x_m$.

\begin{enumerate} 
\item  Let $\mathcal N$ be an arbitrary ${\mathcal L}$-structure. Then $T_{\rm ind}( {\mathcal M}, {\mathcal N} )$ is the set of $m$-tuples $(x'_1, \ldots, x'_m) \in \vert {\mathcal N} \vert^k$ such that the map which takes $x_i$ to $x'_i$ for each $i$ is an embedding from $\mathcal M$ to $\mathcal N$.

\item  Let $\mathfrak F$ be an $\mathcal L$-limit. Then $T_{\rm ind}( {\mathcal M}, {\mathfrak F})$ is the set of tuples $\vec y \in [0,1]^{r([m], r_{\rm max})}$ such that the map which takes $x_i$ to $i$ for each $i$ is an embedding from $\mathcal M$ to ${\mathcal N}( {\mathfrak F}, m, \vec y)$.   

\end{enumerate}
\end{definition}
   Of course the definitions of these sets depend on the choice of an enumeration for $\vert {\mathcal M} \vert$, but this is irrelevant in applications. 
  We note that when $\mathcal N$ is finite, $t_{\rm ind}( {\mathcal M}, {\mathcal N})$ is just the normalised counting measure of $T_{\rm ind}( {\mathcal M}, {\mathcal N})$.

\begin{proof}  

 By Fact \ref{Razfact}, it will suffice to show that for each $m$ and each structure $\mathcal M$ with $\Vert {\mathcal M} \Vert = m$, $P( \boldsymbol{\mathcal N}( {\mathfrak E}, m)  \simeq {\mathcal M} ) = \Phi_p( {\mathcal M} )$.
 We will actually show that the probability $p$ that a random bijection from $\vert {\mathcal M} \vert$ to $\vert \boldsymbol{\mathcal N}( {\mathfrak E}, m) \vert$ is an isomorphism is $\Phi_{t_{\rm ind}}( {\mathcal M} )$, from which the desired conclusion follows immediately using Fact \ref{ptfact}.  
 To show that $p = \Phi_{t_{\rm ind}}( {\mathcal M} )$, we  will define measurable sets $F_1$ and $F_2$ such that the measure of $F_1$ is $\Phi_{t_{\rm ind}}({\mathcal M})$ and the measure of $F_2$ is $p$, and then use the ideas of subsection \ref{EStheory} to show that $F_1$ and $F_2$ have the same measure.

 Let $F_1 = T_{\rm ind}( {\mathcal M}, {\mathcal N}_\infty)$;    by Claim \ref{claimone}  and the  definition of $\mu_{[m]}$ as the standard part of the normalised counting measure on $X_\infty^m$, it is immediate that $\Phi_{t_{\rm ind}}({\mathcal M}) = \mu_{[m]}(F_1)$.

 Let $F_2 = T_{\rm ind}({\mathcal M}, {\mathfrak E})$; it is easy to see that the probability $p$ that a bijection between $\vert {\mathcal M} \vert$ and $\vert \boldsymbol{\mathcal N}( {\mathfrak E}, m)  \vert$ selected uniformly at random is an isomorphism is $\lambda(F_2)$, where $\lambda$ is Lebesgue measure on  $[0,1]^{r([m],r_{\rm max} )}$. 

Let $\vert {\mathcal M} \vert = \{ x_1, \ldots, x_m \}$. By Fact \ref{liftingfact}, the separable system $\{ l_A : A \in r([r_{\rm max}]) \}$ and its realisation $\phi: X_\infty^{r_{\rm max}} \rightarrow [0, 1]^{r([r_{\rm max}])}$ can be lifted, to obtain a family of algebras $\{ l_A : A \in r([m], r_{\rm max}) \}$ and a map $\phi: X_\infty^m \rightarrow [0, 1]^{r([m], r_{\rm max})}$.
 
 By definition, $F_1$ is the set of tuples $(x'_1, \ldots, x'_m) \in X_\infty^{\underline m}$ such that for every $i \in [n]$,  and every $r_i$-tuple $(a_1, \ldots, a_{r_i})$ of elements of $[m]$, $R_i^{\mathcal M}(x_{a_1}, \ldots, x_{a_{r_i}})$ if and only if $R_i^{ {\mathcal N}_\infty }(x'_{a_1}, \ldots, x'_{a_{r_i}})$.
 Recalling how we coded relations into families of directed hypergraphs, it is routine to check that $F_1$ is the set of tuples  $(x'_1, \ldots, x'_m) \in X_\infty^{\underline m}$ such that for every $(i, p) \in \Index$ (say with $\Vert p \Vert = t$) and every $t$-tuple $(b_1, \ldots, b_t) \in [m]^{\underline t}$, $\DHyp_p^{ {\mathcal M}, i}(x_{b_1}, \ldots, x_{b_t})$ if and only if $\DHyp_p^{ {\mathcal N}_\infty, i}(x'_{b_1}, \ldots, x'_{b_t})$.

 By definition, $F_2$ is the set of tuples $\vec y$ of reals such that  for every $i \in [n]$, and every $r_i$-tuple $(a_1, \ldots, a_{r_i})$ of elements of $[m]$, $R_i^{\mathcal M}(x_{a_1}, \ldots, x_{a_{r_i}})$ if and only if $R_i^{ {\mathcal N}({\mathfrak E}, m, {\vec y}) }(a_1, \ldots, a_{r_i})$.
 In terms of the coding by directed hypergraphs, this condition states that for every $(i, p) \in \Index$ with $\Vert p \Vert = t$ say, and every $t$-tuple $(b_1, \ldots, b_t) \in [m]^{\underline t}$, $\DHyp_p^{ {\mathcal M}, i}(x_{b_1}, \ldots, x_{b_t})$ if and only if $\DHyp_p^{ {\mathcal N}({\mathfrak E}, m, {\vec y}), i}(b_1, \ldots, b_t)$.
  
   By the definition of the structure  ${\mathcal N}({\mathfrak E}, m, {\vec y})$, $\DHyp_p^{ {\mathcal N}({\mathfrak E}, m, {\vec y}), i}(b_1, \ldots, b_t)$ if and only if  $\vec z \in E_p^i$, where  $\vec z = (z_B : B \in r([t]))$ is given by  $z_B = y_{ \{b_l : l \in B \} }$.
Recalling that we chose the sets $E_p^i$ such that  $\DHyp_p^{ {\mathcal N}_\infty, i}$ and $\phi^{-1} [E_p^i]$ differ by  a null set, routine calculation now shows that $F_1$ and $\phi^{-1}[F_2]$ differ by a null set, so that (since the lifted version of $\phi$ is measure preserving) $\Phi_{t_{\rm ind}}( {\mathcal M}) = \mu_{[k]}(F_1) = \lambda(F_2) = p$ as required.
\end{proof} 
   We make a  few concluding remarks:
\begin{itemize}
\item   Adapting proofs by Lov\'{a}sz and Szegedy \cite{LS04} for graph limits, one can show that in fact for any ${\mathcal L}$-limit $\mathfrak F$ the sequence 
$( \boldsymbol{\mathcal N}( {\mathfrak F}, j^2) )$ is convergent.

\item  From the results above, the reader who is familiar with flag algebras will  see that our notion of ${\mathcal L}$-limit  and Razborov's notion of flag algebra homomorphism (in the special case of the zero flag and the empty theory) are intertranslatable. One could see them as two ways of describing a notion of ``random sequence of ${\mathcal L}$-structures''.   

\item  It is easy to see that any quantity of the form $\Phi_{p/t/t_{\rm ind}}( {\mathcal M} )$ can be expressed as a suitable integral. In general these integral expressions are quite complicated, we will give a few simple examples in the Appendix.
\end{itemize}

\subsection{Hyperpartitions} \label{hypsec}
     We sketch some ideas of Elek and Szegedy from their work \cite{ES08b} on hypergraph limits. To motivate the definitions in this section, we consider a separable realisation $\phi : X_\infty^t \rightarrow [0,1]^{r([t])}$. Given $l > 0$, we divide the unit interval into $l$ subintervals of equal measure by defining $I^a_l = [\frac{a-1}{l}, \frac{a}{l})$ for $1 \le a < l$, $I^l_l = [\frac{l-1}{l}, 1]$.
    
   Given a tuple $\vec e  = (e_i : i \in I) \in [l]^I$ for some index set $I$, we let $\Cube(\vec e)$ be the ``hypercube'' $\prod_{i \in I} I^{e_i}_l \subseteq [0,1]^I$. One key point here is that any measurable subset of $[0,1]^I$ can be approximated arbitrarily well by sets which are unions of hypercubes: this plays an important role in the proofs of Removal and Regularity. Of course formally the definition of $\Cube(\vec e)$ depends on the value $l$, but in practice this should always be clear from the context. For every tuple $\vec x \in X_\infty^t$ there is a unique tuple $e(\vec x) \in [l]^{r[t]}$ such that $\phi(\vec x) \in \Cube(e(\vec x))$. In particular, there is a natural equivalence relation according to which tuples $\vec x, {\vec  x}' \in X_\infty^t$ are equivalent if and only if $e({\vec x}) = e({\vec x}')$.
  
     Recall now the description of $\phi$ in terms of the ``levels'' $\phi^j$ from subsection \ref{EStheory}. For each $j$ and $e$, let $H^j_e = \{ \{ x_1, \ldots, x_j \}  \in X_\infty^j : \phi^j(x_1, \ldots, x_j) \in I^e_l \}$, where we note that $H^j_e$ is a family of sets rather than tuples, and  the definition makes sense  because $\phi^j$ is invariant under permutation of its arguments. 
 
   It is now immediate from the definitions that:
\begin{itemize}

\item $e(\vec x)_A$ is the unique $e \in [l]$ such that $\{ x_j : j \in A \} \in H^{\vert A \vert}_e$.

\item If $\vec x$ and ${\vec x}'$ are $t$-tuples in $X_\infty^{\underline t}$, $\vec x$ and ${\vec x}'$ are equivalent (in the sense that $e(\vec x) = e({\vec x}')$) if and only if for all $A \in r([t])$, $\{ x_j : j \in A \}$ and $\{ x_j' : j \in A \}$ lie in the same cell of the partition of $X_\infty^{\vert A \vert}$ given by $\{ H^{\vert A \vert}_e: e \in [l] \}$.

\end{itemize}

    These considerations motivate the abstract definition of a {\em hyperpartition}, and the related notions of {\em equitability} and  {\em $\delta$-equitability}.  

\begin{definition}  Let $X$ be a set.
\begin{itemize}
\item  A {\em $(t, l)$-hyperpartition for $X$} is a family $\mathscr{H} = (H^j_e : j \in [t], e \in [l])$ such that for each $j \in [t]$, $(H^j_e : e \in [l])$ forms a partition
 of $X^j$. 
\item Given such a hyperpartition, we let $h^j_e = \{ (x_1, \ldots x_j) : \{ x_1, \ldots x_j \} \in H^j_e$.
\item  If $X$ is equipped with a measure $\nu$ and for each $j$ and $e$ the set $h^j_e$ is $\nu^j$-measurable (where $\nu^j$ is product measure on $X^j$), then $\mathscr{H}$ is {\em equitable} if and only if $\nu^j(h^j_e) = \nu^j(h^j_{e'})$ for all $e, e'$,  and {\em $\delta$-equitable} if and only if $\vert \nu^j(h^j_e) - \nu^j(h^j_{e'}) \vert < \delta$ for all $e, e'$. 
\end{itemize}
  \end{definition} 
  It should be clear that the family of partitions defined above from a separable realisation $\phi$ constitutes an equitable hyperpartition for $X_\infty$. In what follows we will denote this hyperpartition by $\mathscr{H}^\phi_\infty = (H^{\phi,j}_{\infty,e}: j \in [t], e \in [l])$.

    Let $\mathscr{H}$ be a  $(t, l)$-hyperpartition for some set $X$, let $u \le t$ and let $\vec x = (x_1, \ldots x_u) \in X^{\underline u}$. Then we define $e^{\mathscr{H}}(\vec x)  \in [l]^{r([u])}$ by setting $e^{\mathscr{H}}(\vec x)_A = b$ for the unique $b$ such that $\{ x_i : i \in A \} \in H^u_b$.
  We note that $(H^j_a: j \in [u], a \in [l])$ is a $(u, l)$-hyperpartition for $X$, and it suffices to determine the values of $e^{\mathscr{H}}$ on $u$-tuples; in the case of the hyperpartition $\mathscr{H}^\phi_\infty$, this $(u, l)$ hyperpartition coincides with the one defined from the lowered version of $\phi$. 
    Given $\vec e \in [l]^{r[u]}$, we further define $\Cell^{\mathscr{H}}(\vec e)$ to be the set of $u$-tuples $\vec x$ such that $e^{\mathscr{H}}(\vec x) = \vec e$. 
  Note that by definition, $\Cell^{\mathscr{H}^\phi_\infty}(\vec e) = \phi^{-1}[\Cube(\vec e)]$, and that $\Cell^{\mathscr{H}}(\vec e)$ is a $u$-uniform directed hypergraph.
  
  For use in the proofs of Removal and Regularity, we elaborate our discussion of $\mathscr{H}^\phi_\infty$ by relating it to the ultrapower construction. Each of the sets $h^j_{\infty, a}$ is measurable and so it differs on a set of measure zero from a set in $V^*$; we may therefore choose a new equitable hyperpartition $\mathscr{H}_\infty \in V^*$ such that $h^j_{\infty,e} \Delta h^{\phi, j}_{\infty,e}$ is null for all $e$ and $j$.
    We now choose $({\mathscr H}_k)$ so that ${\mathscr H}_\infty = [{\mathscr H}_k]$, and ${\mathscr H}_k$ is a $(t, l)$-hyperpartition on $X_k = \vert {\mathcal N}_k \vert$ for each $k$. For use in the proof of Regularity, we note that we may find numbers $\delta_k$ such that ${\mathscr H}_k$ is $\delta_k$-equitable and $\lim_{k \rightarrow U} \delta_k = 0$. 
  To lighten the notation we write $e_k$ for the function $e^{\mathscr{H}_k}$, and similarly $\Cell_k$ for $\Cell^{\mathscr{H}_k}$. It is routine to verify that $[\Cell_k(\vec e)] = \Cell^{\mathscr{H}_\infty }(\vec e)$, which differs on a null set from   $\Cell^{\mathscr{H}^\phi_\infty}(\vec e) = \phi^{-1}[\Cube(\vec e)]$.

\subsection{Removal}  \label{removalsection}
   Given a finite set $X$, we introduce a metric $d$ on the set of $\mathcal L$-structures with underlying set $X$. The key point is that $d ({\mathcal M}, {\mathcal N})$ measures the difference between the edge sets of the various directed hypergraphs comprising $\Code({\mathcal M})$ and $\Code({\mathcal N})$. 
 We start by defining $d_p^i( {\mathcal M}, {\mathcal N} )$ to be $\vert \DHyp_p^i( {\mathcal M} ) \Delta  \DHyp_p^i( {\mathcal N} ) \vert / \vert X \vert^{\Vert p \Vert}$ for all $(i, p) \in \Index$, then let $d( {\mathcal M}, {\mathcal N}) = \sum_{(i, p) \in \Index} d_p^i( {\mathcal M}, {\mathcal N} )$.
    We can now state and prove a version of the Strong Removal Theorem. 
\begin{theorem} \label{removal} 

 Let $\mathcal F$ be a (possibly infinite) set of finite ${\mathcal L}$-structures.
  For every $\epsilon > 0$ there exist $\delta > 0$ and $m$ such that for all sufficiently large finite ${\mathcal L}$-structures $\mathcal N$,  {\bf if} $p({\mathcal M}, {\mathcal N}) < \delta$ for all ${\mathcal M} \in {\mathcal F}$ with $\Vert {\mathcal M} \Vert \le m$, {\bf then} there is ${\mathcal N}'$ such that $\vert {\mathcal N} \vert = \vert {\mathcal N}' \vert$,   $d( {\mathcal N}, {\mathcal N}') < \epsilon$ and $p( {\mathcal M}, {\mathcal N}') = 0$ for all ${\mathcal M} \in {\mathcal F}$.

\end{theorem}

\begin{proof}  If not, then we fix an $\epsilon > 0$ for which the claim of the theorem fails.
  We use this to construct an increasing sequence $({\mathcal N}_k)$ of finite structures  such that for all $k$:

\begin{itemize}

\item $p({\mathcal M}, {\mathcal N}_k) < 1/k$ for all ${\mathcal M} \in {\mathcal F}$ with $\Vert {\mathcal M} \Vert \le k$.

\item  There is no ${\mathcal N}'$ such that
  $\vert {\mathcal N}_k \vert = \vert {\mathcal N}' \vert$,   $d( {\mathcal N}_k, {\mathcal N}') < \epsilon$
  and $p( {\mathcal M}, {\mathcal N}') = 0$ for all ${\mathcal M} \in {\mathcal F}$.

\end{itemize} 
   The strategy of the proof is as follows: we will produce for some large $k$, structures ${\mathcal N}^*_k$ and ${\mathcal N}^\dag_k$ such that $d( {\mathcal N}_k, {\mathcal N}^*_k) < \epsilon/2$, $d( {\mathcal N}^*_k, {\mathcal N}^\dag_k) < \epsilon/2$, and $p({\mathcal M}, {\mathcal N}^\dag_k) = 0$ for all ${\mathcal M} \in {\mathcal F}$. This will give an immediate contradiction.
   Thinning it out  if necessary, we may assume that the sequence $({\mathcal N}_k)$ is convergent and may form the associated function $\Phi_p$. By construction, 
$\Phi_p({\mathcal M} ) = 0$ for all ${\mathcal M} \in F$. 
    We now form the ultrapower ${\mathcal N}_\infty$, and define a separable realisation $\phi$ and ${\mathcal L}$-limit $\mathfrak E$ as in subsection \ref{limsec}.
    Since $\Phi_p( {\mathcal M} ) = 0$ for all ${\mathcal M} \in F$, it follows (as in the proof of Theorem \ref{limittheorem}) that $T_{\rm ind}({\mathcal M}, {\mathfrak E})$ is a null set for all ${\mathcal M} \in F$.
    We recall that $\mathfrak E$ is a family of Lebesgue measurable sets $E_p^i$, where $p$ is a partition of $[r_i]$ into some number  $t = \Vert p \Vert$ of pieces and $E_p^i \subseteq [0,1]^t$.  We choose $\epsilon' > 0$ small enough that $\vert  \Index \vert \epsilon' < \epsilon/2$, and then choose an integer $l$ so large that every set $E_p^i$ can be approximated with error $\epsilon'$ by a union of hypercubes which are products of intervals of the form $I^e_l$: that is for every $i$ and $p$ there is a set  $X_p^i \subseteq [l]^{r([\Vert p \Vert])}$ such that if we set $E_p^{*, i} = \bigcup_{\vec e \in X_p^i} \Cube(\vec e)$ then $\lambda(E_p^i \Delta E_p^{*, i}) < \epsilon'$. 

    With these definitions in place, we can now define $(r_{\rm max}, l)$-hyperpartitions $\mathscr{H}_\infty^\phi$, $\mathscr{H}_\infty$ and $\mathscr{H}_k$ exactly as in subsection \ref{hypsec}.
  Let $l_k = \Vert {\mathcal N}_k \Vert$, so that we may assume $\vert {\mathcal N}_k \vert = [l_k]$. We  use the hyperpartition $\mathscr{H}_k$ and the sets $X_p^i$ to define a structure ${\mathcal N}^*_k$ with $\vert {\mathcal N}^*_k \vert = \vert {\mathcal N}_k \vert$.
 
  For each $i \in [n]$ and each partition $p$ of $[r_i]$, we  use $\mathscr{H}_k$ to partition $\Vert {\mathcal N}_k \Vert^{\Vert p \Vert}$ into cells of the form $Cell_k(\vec e)$ for $e \in [l]^{r[t]}$. Then we form sets $\DHyp^{*, i}_{k,p} = \bigcup_{\vec e \in X_p^i} \Cell_k(\vec e)$, and finally we decode the resulting system of directed hypergraphs  by setting ${\mathcal N}^*_k = \Decode(\DHyp^{*, i}_{k,p} : (i, p) \in \Index)$.

    By a routine application of {\L}o{\'s}' theorem, the ultraproduct ${\mathcal N}^*_\infty$ is the result of decoding the  hypergraphs obtained in a similar way from ${\mathcal H}_\infty$.  Chasing through the definitions, this means that for each $i$ and $p$, the directed hypergraph $\DHyp^{ {\mathcal N}^*_\infty, i}_p$ is $\bigcup_{\vec e \in X_p^i} \Cell^{ {\mathscr H}_\infty }(\vec e)$.
  Recalling  that $\phi$ is measure preserving, and $\lambda(E_p^i \Delta E_p^{*,i}) < \epsilon'$, we see that $\mu( \phi^{-1}[E_p^i] \Delta \phi^{-1}[E_p^{*,i}]) < \epsilon'$.
  Since $\phi^{-1}(E_p^i) \Delta \DHyp_p^{ {\mathcal N}_\infty, i}$ is null, $\phi^{-1}[E_p^{*, i}] = \bigcup_{\vec e \in X_p^i} \Cell^{ \mathscr{H}^\phi_\infty}(\vec e)$, and  $\Cell^{ {\mathscr H}_\infty }(\vec e) \Delta \Cell^{ {\mathscr H}^\phi_\infty }(\vec e)$ is null,  we see that 
\[
   \mu (\DHyp^{ {\mathcal N}^*_\infty, i}_p \Delta \DHyp^{ {\mathcal N}_\infty, i}_p) < \epsilon'
\]
for all $i$ and $p$, so that easily $d( {\mathcal N}^*_k, {\mathcal N}_k) < \epsilon/2$ for $U$-almost every $k$.  
  
    Both $\mathfrak E$ and ${\mathfrak E}^*$ are ${\mathcal L}$-limits, but (following an idea of Elek and Szegedy which they dub ``hyperpartition sampling'') we will use them in a slightly different way from that described in subsection \ref{limsec} to define random structures $\boldsymbol{\mathcal N}_k$ and $\boldsymbol{\mathcal N}^*_k$.
  Before making the definitions we warn the reader that $\boldsymbol{\mathcal N}^*_k$ is a rather trivial random structure in that (although we will shortly give a more complicated definition) it is always equal to ${\mathcal N}^*_k$; the point in defining $\boldsymbol{\mathcal N}^*_k$ in this seemingly obtuse way is to facilitate comparison with $\boldsymbol{\mathcal N}_k$.
    We define $\boldsymbol{\mathcal N}_k$ using the same function $\vec y \mapsto {\mathcal N}( {\mathfrak E}, l_k, \vec y)$  that was used in the definition of $\boldsymbol{\mathcal N}( {\mathfrak E}, l_k)$  from subsection \ref{limsec}; the key difference is that we choose values of $\vec y$ uniformly from  the hypercube  $\Cube(\vec c) \subseteq [0,1]^{ r([l_k], r_{\rm max}) } $, where  $\vec c = (c_A : A \in r([l_k], r_{\rm max}) )$ and $c_A$ is defined to be the unique  $j \in [l]$ such that
  $A \in H^{\vert A \vert}_{k, j}$.
  Similarly we define $\boldsymbol{\mathcal N}^*_k$ to  be the value of ${\mathcal N}( {\mathfrak E}^*, l_k, \vec y)$ where again $\vec y$ is chosen uniformly at random from $\Cube(\vec c)$.
   Since  $T_{\rm ind}({\mathcal M}, {\mathfrak E})$ is a null set for all ${\mathcal M} \in {\mathcal F}$ and $\mathcal F$ is countable, it is easy to see that almost surely $p({\mathcal M}, \boldsymbol{\mathcal N}_k) = 0$ for all ${\mathcal M} \in {\mathcal F}$. 
    We claim that ${\mathcal N}( {\mathfrak E}^*, l_k, \vec y) = {\mathcal N}^*_k$ for all $\vec y \in \Cube(\vec c)$, which we will verify by inspecting the computation of the associated directed hypergraphs. Let us fix $\vec y \in \Cube(\vec c)$, $(i, p) \in \Index$ with $\Vert p \Vert = t$ say, and $(a_1, \ldots, a_t) \in [l_k]^{\underline t}$: then by definition $(a_1, \ldots, a_t) \in \DHyp_p^{{\mathcal N}( {\mathfrak E}^*, l_k, \vec y), i}$ if and only if $(y_{ \{ a_i : i \in B \} } : B \in r([t]) ) \in E^{*, p}_i$ if and only if $e_k(a_1, \ldots, a_t) = (c_{ \{ a_i : i \in B \} } : B \in r([t]) ) \in X^p_i$ if and only if $(a_1, \ldots, a_t) \in \DHyp_p^{ {\mathcal N}^*_k, i}$.
   Next we claim that $\lim_{k \rightarrow U} {\mathbb E}(d(\boldsymbol{\mathcal N}_k, \boldsymbol{\mathcal N}^*_k)) < \epsilon/2$. 
 Fix $k$, and    $(i, p) \in \Index$ and let $t = \Vert p \Vert$. For each $\vec a = (a_1, \ldots, a_t) \in [l_k]^{\underline t}$, let $\boldsymbol{Y}(\vec a)$ be the indicator function of the event ``$\vec a \in \DHyp_p^{\boldsymbol{\mathcal N}_k, i} \Delta \DHyp_p^{\boldsymbol{\mathcal N}^*_k, i}$''.
 The expected value of $\boldsymbol{Y}(\vec a)$ is the probability  that the tuple  $(y_{ \{ a_i : i \in B \} } : B \in r([t]) )$ lies in $E^p_i \Delta E^{*, p}_i$; since $y_{ \{ a_i : i \in B \} }$ is chosen uniformly from the interval $I^{y_{ \{ a_i : i \in B \} } }_l$, we see easily that  $${\mathbb E}( \boldsymbol{Y}(\vec a) )  = l^{2^t - 1} \lambda(E^p_i \Delta E^{*, p}_i \cap B(\vec a) ),$$ where $B(\vec a) = \Cube(c_{ \{ a_i : i \in A \} } : A \in r([t]) )$.  
So 
\begin{eqnarray*} 
\mathbb E(\vert \DHyp_p^{\boldsymbol{\mathcal N}_k, i} \Delta \DHyp_p^{\boldsymbol{\mathcal N}^*_k,i} \vert/l_k^{2^t -1}) 
 &  =  &  \sum_{\vec a \in [l_k]^{\underline t} }  \lambda(E^p_i \Delta E^{*, p}_i \cap B(\vec a) ) \\
{} & =  & \sum_{\vec e \in [l]^{r[t]}} \vert \Cell_k(\vec e) \vert (l/l_k)^{2^t -1} \lambda(E^p_i \Delta E^{*, p}_i \cap \Cube(\vec e) ).\\
\end{eqnarray*}

 Since the sequence $( \Cell_k(\vec e) )$ represents a set of measure $l^{1 - 2^t}$, it follows that $\lim_{k \rightarrow U} \vert \Cell_k(\vec e) \vert (l/l_k)^{2^t -1} = 1$. Hence 
\begin{eqnarray*}
 \lim_{k \rightarrow U} {\mathbb E}(\vert \DHyp_p^{\boldsymbol{\mathcal N}_k, i} \Delta \DHyp_p^{\boldsymbol{\mathcal N}^*_k,i} \vert/l_k^{2^t -1}) 
 & = & \sum_{\vec b \in [l]^{r[t]}}  \lambda(E^p_i \Delta E^{*,p}_i \cap \Cube(\vec b) ) \\ {} &  = & \lambda(E^p_i \Delta E^{*,p}_i) < \epsilon'. 
\end{eqnarray*}
   By the choice of $\epsilon'$, it follows that $\lim_{k \rightarrow U} {\mathbb E}(d(\boldsymbol{\mathcal N}_k, \boldsymbol{\mathcal N}^*_k)) < \epsilon/2$. 

 Since  almost surely  $p(M, \boldsymbol{\mathcal N}_k) = 0$ for all $M \in F$ and all $k$, it follows from the claim of the last paragraph that for $U$-many $k$ there is a structure ${\mathcal N}^\dag_k$ such that $d( {\mathcal N}^\dag_k, {\mathcal N}^*_k) < \epsilon/2$ and $p(M, {\mathcal N}^\dag_k) = 0$ for all $M \in F$. Since  $d( {\mathcal N}^*_k, {\mathcal N}_k) < \epsilon/2$ for $U$-many $k$, and $d$ is a metric, this implies that there is  some $k$ such that $d( {\mathcal N}^\dag_k, {\mathcal N}_k) < \epsilon$
 and $p(M, {\mathcal N}^\dag_k) = 0$ for all $M \in F$. This is a contradiction.
\end{proof}

 It is easy to see that Theorem \ref{removal} implies the form of removal with one  omitted structure and homomorphic images:
\begin{theorem} Let $\mathcal M$ be a finite $\mathcal L$-structure.
  For every $\epsilon > 0$ there exists $\delta > 0$ such that for all large ${\mathcal L}$-structures $\mathcal N$,  {\bf if} $t({\mathcal M}, {\mathcal N}) < \delta$ {\bf then} there is ${\mathcal N}'$ such that  $\vert {\mathcal N} \vert = \vert {\mathcal N}' \vert$,   $d( {\mathcal N}, {\mathcal N}') < \epsilon$  and $p( {\mathcal M}, {\mathcal N}') = 0$.
\end{theorem}
  Theorem \ref{removal} can also be applied to prove testability for hereditary properties of structures, using standard arguments.

\subsection{Regularity} \label{regularity} 
\begin{definition}

A \emph{cylindric intersection set} $L$ in $X^{A}$ is a set $L = \bigcap_{\{B : B \subsetneq A\}} \pi_B^{-1}(Y_B)$, where each $Y_B \subseteq X^B$ is measurable.

\end{definition}

If ${L}$ is a cylindric intersection in the ultraproduct $X_{\infty}^{[r]}$ then ${L} \in \sigma_{[r]}^*$.

In the case that $r=2$, every cylindric intersection in $X^{[2]}$ is a measurable rectangle, whether X is a finite set or the ultraproduct.
\begin{definition}

\label{epsilonregular}

A directed $r$-uniform hypergraph $H$ on X is \emph{$\epsilon$-regular} if for any cylindric intersection set L in $X^{[r]}$ with $\mu(L) \geq \epsilon$, $ | \mu(H \cap L) - \mu(H)\mu(L) | < \epsilon$. 

\end{definition}
If $H$ is $\epsilon$-regular for all $\epsilon > 0$, then $\mu(H \cap L) = \mu(H)\mu(L)$ for any cylindric intersection $L$. \\
We can now state and prove a version of the Regularity Lemma.

\begin{theorem} \label{regularitylemma} Given $\epsilon > 0$ and $d \in \mathbb{N}$, there exist $D, N \in \mathbb{N}$ such that for any $\mathcal{L}$-structure ${\mathcal M}$ with $\Vert {\mathcal M} \Vert \geq  N$, there exists an $\epsilon$-equitable $(r_{\rm max}, l)$-hyperpartition ${\mathscr H}$ on $\vert {\mathcal M} \vert$ such that 

\begin{itemize}

\item $d \leq l \leq D$

\item Each $h^j_e$ is $\epsilon$-regular.

\item For each $(i, p) \in \Index$, there exist $m$ and tuples $\vec{e}_1, \ldots, \vec{e}_{m}$ such that  $\mu(\DHyp_p^{\mathcal{M},i} \Delta \bigcup_{i = 1}^{m} \Cell^{\mathscr{H}}(\vec{e}_i)) < \epsilon$.
\end{itemize}
\end{theorem}
\begin{proof}

If not, then we fix an $\epsilon > 0$ for which the claim of the theorem fails.
  We use this to construct an increasing sequence $({\mathcal N}_k)$ of finite structures such that for all $k$, there is no $\epsilon$-equitable $(r_{\rm max}, l)$-hyperpartition, for $d \leq \ell \leq k$ satisfying above conditions for any ${\mathcal N}_k$.
  
 We now form the ultrapower ${\mathcal N}_\infty$ and define $(r_{\rm max}, l)$-hyperpartitions $\mathscr{H}_\infty^\phi$, $\mathscr{H}_\infty$ and $\mathscr{H}_k$ as in the proof of the Removal Lemma.
  
Each $h_{\infty, e}^j$ is independent of every set $L \in \sigma_{[j]}^*$, that is, $\mu(h_{\infty, e}^j \cap L) = \mu(h_{\infty, e}^j) \mu(L)$.   Then $\mathscr{H}_\infty^\phi$ and  $\mathscr{H}_\infty$ are $\epsilon$-regular as well as equitable for all $\epsilon > 0$.
 Also $\mathscr{H}_\infty$ has the property that for each
$(i,p)\in \Index$, $\DHyp_p^{\mathcal{N},i} = \bigcup_{i = 1}^{m} \Cell^{\mathscr{H}_\infty}(\vec{e}_i)$ for some $m$ and tuples $\vec{e}_1, \ldots, \vec{e}_{m}$.
For $U$-almost every $k$, $\mathscr{H}_k$ is an $\epsilon$-equitable $(r_{\rm max}, l)$-hyperpartition and for each $(i,p)\in \Index$, $\DHyp_p^{\mathcal{N}_k,i} = \bigcup_{i = 1}^{m} \Cell_k(\vec{e}_i)$ for some $m$ and tuples $\vec{e}_1, \ldots, \vec{e}_{m}$. 
By our assumption there must exist $e$ and $j$ such that for $U$-almost every $k$, $h_{e,k}^j$ is not $\epsilon$-regular. Let $(L_k)$ be the sequence of cylindrical intersection sets that witness this fact. Let $L = [L_k] \in \sigma_[j]C^*$. Our assumption implies $h_{e, \infty}^j$ and $L$ are not independent, which is a contradiction.
\end{proof}

\section{Models of a general universal theory} \label{withT}

Let $T$ be a universal theory in the language $\mathcal L$. We need a technical result describing the models of $T$.
\begin{lemma} \label{forbidlemma}  There is a set ${\rm Forbid}^T$ of isomorphism types of finite $\mathcal L$-structures such that for any $\mathcal L$-structure $\mathcal M$, ${\mathcal M} \models T$ if and only if no finite induced substructure of $\mathcal M$ has isomorphism type in ${\rm  Forbid}^T$.
\end{lemma}
   Since the proof is easy and tedious, we have relegated it to the appendix.
It follows immediately that the limit of a convergent sequence of models of $T$ describes a random model of $T$. To be more precise we have :
\begin{theorem} 
Let $( {\mathcal N}_k )$ be a convergent sequence of finite models of $T$, and let $\mathfrak E$ be the limit as constructed in subsection \ref{limsec}. Then almost surely $\boldsymbol{\mathcal N}( {\mathfrak E}, m)$ is a model of $T$ for all $m$. 
\end{theorem}
\begin{proof} For each type $\mathcal M$ in ${\rm Forbid}^T$, $p({\mathcal M}, {\mathcal N}_\infty) = 0$ and so $T_{\rm ind}( {\mathcal M}, {\mathfrak E} )$ is a null set. It follows that almost surely $p({\mathcal M}, \boldsymbol{\mathcal N}( {\mathfrak E}, m)) = 0$ for each ${\mathcal M}$ and $m$, so we are done by $\sigma$-additivity. 
\end{proof} 
   
 The integrals representing $\Phi_{p/t/t_{\rm ind}}( {\mathcal M} )$ in terms of the limit object $\mathfrak E$ sometimes become simpler when $\mathfrak E$ is a limit of models of $T$. Some examples appear in Section \ref{examples} of the Appendix.  

\begin{theorem} Let $T$ be a universal theory in the finite language $\mathcal L$.

 Let $F$ be a (possibly infinite) set of finite models of $T$. For every $\epsilon > 0$ there exist $\delta > 0$ and $m$ such that for all sufficiently large finite models $\mathcal N$ of $T$,  {\bf if} $p({\mathcal M}, {\mathcal N}) < \delta$ for all ${\mathcal M} \in F$ with $\Vert {\mathcal M} \Vert \le m$, {\bf then} there is  a model ${\mathcal N}'$ of $T$ such that $\vert {\mathcal N} \vert = \vert {\mathcal N}' \vert$,   $d( {\mathcal N}, {\mathcal N}') < \epsilon$ and $p( {\mathcal M}, {\mathcal N}') = 0$ for all ${\mathcal M} \in F$.
\end{theorem}
\begin{proof} Let $F^+ = F \cup  {\rm Forbid}^T$ and apply Theorem \ref{removal} to $F^+$. 
\end{proof}

   Note that the proof gives a bit more: there exist $m$ and a finite subset $T_0$ of $T$ such that any large model of $T_0$ with the property that $p({\mathcal M}, {\mathcal N}) < \delta$ for all  ${\mathcal M} \in F$ with $\Vert {\mathcal M} \Vert \le m$ is within $\epsilon$ of some model of $T$.

\section{Final Remarks}
 We note that our Strong Removal Lemma for models of universal theories allows for the extension of property testing results to such structures. We have also successfully investigated the existence of limits, regularity and removal lemmas for general weighted structures. This work, which also uses the theory of measures on ultraproducts,
 will appear in a subsequent paper.  If $(G_n)$ is a sequence of graphs where the density of edges tends to zero (for example, a sequence of graphs of bounded degree) then the limit graphon is identically zero. There is a cultivated theory of limits for sequences of graphs of bounded degree, where the limit object (a {\em graphing}) is defined in a different way.  It is not clear whether this theory can be extended in a meaningful way to more general classes of structures. 
 
\appendix
\section{Some examples} \label{examples}
\subsection{Structures with one binary relation} 
Let $\mathcal{L}$ only contain one binary relation symbol $R$.
If $\mathcal M$ is an ${\mathcal L}$-structure, then $\Code(\mathcal{M}) = (\DHyp_{\{ \{1\}\{2\} \}}^{\mathcal{M}},  \DHyp_{ \{ \{1,2\} \} }^{\mathcal{M}}  )$, which amounts to saying that we  can view $\mathcal M$ as a digraph with loops and we code it by separately recording the set of directed edges between distinct points and the set of vertices
 with a loop.  
 An $\mathcal{L}$-limit $\mathfrak{F}$ is of the form $(F_{\{ \{1\}\{2\} \}},  F_{ \{ \{1,2\} \} }  )$ where $F_{\{ \{1\}\{2\} \}}$ is a measurable subset of $[0,1]^{r[2]}$ and $F_{\{ \{1, 2\} \}}$ is a measurable subset of $[0,1]^{r[1]}$. 

Let $\mathcal{M}$ be a binary relation with $\Vert \mathcal{M} \Vert = n$, where we  assume for convenience that $\vert \mathcal{M} \vert = [n]$.  As in subsection \ref{limsec} we can compute $\Phi_t(\mathcal{M})$ from $\mathfrak F$ by finding the measure of a certain subset of $[0,1]^{r([n], 2)}$. To give a formula for the measure of this set in a palatable form, let $f_p$ be the characteristic function of the measurable set $F_p$ in the limit $\mathcal{F}$. Then $\Phi_t(\mathcal{M})$ is the integral of  a certain product $F(\vec x) G(\vec x)$ over $[0,1]^{r([n], 2)}$: 

\begin{itemize}

\item $F(\vec x)$ is the product of terms of the form $f_{\{ \{1\}\{2\} \}}(x_{\{i\}}, x_{\{j\}},x_{\{i,j\}})$, taken over all pairs  $(i, j)$ where $i \neq j$ and there is  an edge from $i$ to $j$.

\item $G(\vec x)$ is the product  of terms of the form $f_{ \{ \{1,2\} \}}(x_{\{i\}})$, taken over all $i$ such that there is a loop at $i$. 

\end{itemize}   
The formula for $\Phi_{t_{\rm ind}}(\mathcal{M})$ is very similar in form, but the functions $F(\vec x)$ and $G(\vec x)$ are more complex in this case:
\begin{itemize}

\item $F(\vec x)$ is the product over pairs $(i, j)$ with $i \neq j$ of terms which have the form  $f_{\{ \{1\}\{2\} \}}(x_{\{i\}}, x_{\{j\}},x_{\{i,j\}})$ if there is an edge from $i$ to $j$, and the form $(1 - f_{\{ \{1\}\{2\} \}}(x_{\{i\}}, x_{\{j\}},x_{\{i,j\}}))$ if there is no edge from $i$ to $j$.  

\item  $G(\vec x)$ is the product over $i$ of terms which have the form $f_{ \{ \{1,2\} \}}(x_{\{i\}})$ if there is a loop at $i$, and the form $(1 - f_{ \{ \{1,2\} \}}(x_{\{i\}}))$ if there is no loop at $i$.

\end{itemize}

\subsection{Digraphs} 
   From our point of view a digraph is just a binary relation $R$ in which $v R v$ is false for all $v$. In syntactic terms digraphs are ${\mathcal L}$-structures which are models of the theory $T_{\rm digraphs} = \{ \forall v \; \neg v R v \}$, while in the language of forbidden substructures they are structures that forbid the substructure which has a single vertex with a loop.   In terms of our coding, digraphs are ${\mathcal L}$-structures $\mathcal M$ such that $\DHyp_{ \{ \{1,2\} \}}^{\mathcal{M}} = \emptyset$. It follows that a limit of digraphs is a limit with   $F_{ \{ \{1,2\} \} } = \emptyset$. In this case, if $\mathcal M$ is an ${\mathcal L}$-structure with some loops then (as we would expect) the integral formula gives the value zero.  If $\mathcal{M}$ is  a digraph, then the integral formula for $\Phi_t(\mathcal{M})$ only includes the first product $F(\vec x)$. This integral can be simplified in a suggestive way, by collecting all the terms which involve each pair of variables. Define four functions $F_i(x, y)$  by setting:

\begin{itemize}

\item   $F_0(x, y) = \lambda ( \{ z : \mbox{ $(x, y, z) \notin F_{\{ \{1\}\{2\} \}}$ and $(y, x, z) \notin F_{\{ \{1\}\{2\} \}}$} \} )$,

\item   $F_1(x, y) = \lambda ( \{ z : \mbox{ $(x, y, z) \in F_{\{ \{1\}\{2\} \}}$ and $(y, x, z) \notin F_{\{ \{1\}\{2\} \}}$}  \} )$,

\item   $F_2(x, y) = \lambda ( \{ z : \mbox{ $(x, y, z) \notin F_{\{ \{1\}\{2\} \}}$ and $(y, x, z) \in F_{\{ \{1\}\{2\} \}}$} \} )$,

\item   $F_3(x, y) = \lambda ( \{ z : \mbox{ $(x, y, z) \in F_{\{ \{1\}\{2\} \}}$ and $(y, x, z) \in F_{\{ \{1\}\{2\} \}}$} \} )$.

\end{itemize} 
 It is immediate from the definitions that $F_0 + F_1 + F_2 + F_3 = 1$, $F_0$ and $F_3$ are symmetric, and $F_1(x, y) = F_2(y, x)$. The integral formula for $\Phi_t(\mathcal{M})$ can now be written as the integral over $[0,1]^2$ of a product (taken over all pairs $(i, j) $ with $i < j$) of terms of the following form:

\begin{itemize}

\item $(F_1 + F_3)(x_{\{i\}}, x_{\{j\}})$ if there is only an edge from $i$ to $j$,   

\item $(F_2 + F_3)(x_{\{i\}}, x_{\{j\}})$ if there is only an edge from $j$ to $i$,   

\item $F_3(x_{\{i\}}, x_{\{j\}})$ if there are edges in both directions between $i$ and $j$.   

\end{itemize}
 This calculation has recovered a version of the notion of ``digraph limit'' due to Offner and Pikhurko \cite{Offner}. 

The integral formula for $\Phi_{t_{\rm ind}}(\mathcal{M})$ can be similarly written as the integral over $[0,1]^2$ of a product (taken over all pairs $\{ i, j \}$ with $i < j$) of terms of the following form:

\begin{itemize}

\item $F_0(x_{\{i\}}, x_{\{j\}})$ if there is no edge in either direction between $i$ and $j$,

\item $F_1(x_{\{i\}}, x_{\{j\}})$ if there is only an edge from $i$ to $j$,   

\item $F_2(x_{\{i\}}, x_{\{j\}})$ if there is only an edge from $j$ to $i$,   

\item $F_3(x_{\{i\}}, x_{\{j\}})$ if there are edges in both directions between $i$ and $j$.   

\end{itemize}
\subsection{Graphs and hypergraphs} 

A graph is just a symmetric digraph, or a model of $T_{\rm graphs} = \{ \forall v \; \neg v R v, \forall v \; \forall w \; (v R w \implies w R v)  \}$. In the language of forbidden substructures, graphs are structures which forbid the single vertex with a loop, and a pair of vertices with an edge going only one way. It is easy to see that (in our formulation) the limit of a sequence of graphs will have $F_{ \{ \{1,2\} \} } = \emptyset$, while $F_{\{ \{1\}\{2\} \}}$ is (without loss of generality) a set of triples $(x,y,z)$ which is symmetric in the sense that $(x, y, z) \in F_{\{ \{1\}\{2\} \}} \iff (y, x, z) \in F_{\{ \{1\}\{2\} \}}$.
 
The integral formula for $\Phi_t(M)$ now simplifies further: retaining our notation for digraph limits we have $F_1 = F_2 = 0$, $F_0 + F_3 = 1$ and it is easy to see that we recovered the classical graph limit called graphon in the form of the symmetric measurable function $F_3$. 

\subsection{Posets}

A $\emph{(strict) partial order}$ on $X$ is an irreflexive, transitive binary relation. The theory of partially ordered sets, or \emph{posets} is $T_{\rm posets} = \{\forall v \; \neg vRv, \forall u \; \forall v \; \forall w \; (uRv \wedge vRw \implies uRw) \}$. 
${\rm Forbid}^T = \{ \mathcal{M}_1, \mathcal{M}_2, \mathcal{M}_3, 
\mathcal{M}_4 \}$ where 

\begin{itemize}

\item $\vert \mathcal{M}_1 \vert = [1]$ with $R^{\mathcal{M}_1} = \{ (1,1) \}$,

\item $\vert \mathcal{M}_2 \vert = [2]$ with $R^{\mathcal{M}_2} = \{ (1,2), (2,1) \}$,

\item $\vert \mathcal{M}_3 \vert = [3]$ with $R^{\mathcal{M}_3} = \{(1,2), (2,3)\}$, 

\item $\vert \mathcal{M}_4 \vert = [3]$ with $R^{\mathcal{M}_4} = \{(1,2), (2,3), (3,1)\}$.

\end{itemize}

An $\mathcal{L}$-limit $\mathcal{F}$ has the property that $\Phi_{t_{\text{ind}}}(\mathcal{M}_i) = 0$ for $i \in [4]$. 
A poset is also a digraph. Using the simplified notation for digraph limits, we can say that a poset limit has the additional properties that $F_3=0$ and if $F_1(x,y) > 0$ and $F_1(y,z) > 0$, then $F_1(x,z)=1$.
\medskip

Janson \cite{Janson} defined a notion of poset limit. Such a limit consists of a partial ordering $\prec$ on $[0, 1]$ and  a measurable function $W : [0,1]^2 \to [0,1]$ with the properties that 

\begin{itemize}

\item $W(x,x) = 0$,

\item $W(x,y) > 0 \implies W(y,x) = 0$,

\item $W(x,y) > 0$ and $W(y,z) > 0 \implies W(x,z) = 1$.

\end{itemize}

 The corresponding construction of a random partial ordering of $[n]$ involves randomly choosing $x_i$ and $z_{i, j}$ in $[0,1]$, and then putting $i$ below $j$ if and only if $x_i \prec x_j$ and $z_{i, j} < W(x_i, x_j)$. Given a poset limit in Janson's sense, we can easily retrieve an equivalent poset limit in our sense. 
In the other direction, while we can show that a Janson limit $(\prec, W)$ equivalent to any poset limit $\mathfrak F$ in our sense exists, we do not have an easy algebraic way to find such a Janson limit.

\section{Proof of Lemma \ref{forbidlemma}} \label{forbid}

Let the variable symbols of $\mathcal L$ be listed as $x_1, x_2, \ldots$. For each $m > 0$ let $X^1_m$ be the set of formulae of the form $x_i = x_j$ for $1 \le i \le j \le m$, and let $Y^1_m$ be the set of conjunctions which contain exactly one of $\psi$, $\neg \psi$ for each $\psi \in X^1_m$ and are logically consistent (we note that such conjunctions correspond in the obvious way to partitions of the set $[m]$). 
Let $X^2_m$ be the set of formulae of the form $R_i(z_1, \ldots, z_{n_i})$ where $z_j \in \{x_1, \ldots, x_m \}$, and let $Y^2_m$ be the set of conjunctions which contain exactly one of $\psi$, $\neg \psi$ for each $\psi \in X^2_m.$ 
  Intuitively formulae in $Y^1_m$ are complete descriptions of the equality relation on $\{ x_1, \ldots, x_m \}$, while formulae in $Y^2_m$ are complete descriptions of
the other relations in the language on the same set.
If $\phi$ is a quantifier-free formula which only mentions variables among $x_1, \ldots, x_m$, then by elementary propositional logic it is equivalent to some disjunction of formulae of the form $\psi^1 \wedge \psi^2$ where $\psi^i \in Y^i_m$. Call this the ``normal form'' for $\phi$.

Each axiom $\Psi$ of $T$ is equivalent to a sentence of the form $\forall x_1 \ldots \forall x_m \phi$ where $\phi$ is quantifier-free and is in normal form,  say $\phi = \bigvee_i (\psi^1_i \wedge \psi^2_i)$. Clearly $\Psi$ is equivalent to the set of formulae of the form $\forall x_1 \ldots \forall x_m (\psi \implies \phi)$ with $\psi \in Y^1_m$, where breaking up the formula $\forall x_1 \ldots \forall x_m \phi$ in this way corresponds to a case analysis depending on the equality relation on $\{ x_1, \ldots, x_m \}$. 

For each $\psi \in Y^1_m$,  it is easy to see that $\psi \implies \phi$ is equivalent to $\psi \implies \bigvee_{i, \psi^1_i = \psi} \psi^2_i$. Finally, if $\psi$ corresponds to a partition into $t$ classes, then $\forall x_1 \ldots \forall x_m \psi \rightarrow \phi$ is equivalent (replacing variables in the $j^{\rm th}$ class by $x_j$ and eliminating redundancies) to a formula of the form  $\forall x_1 \ldots \forall x_t (\chi \implies \Phi)$, where $\chi = \bigvee_{1 \le i < j \le t} x_i \neq x_j$ and $\Phi$ is a disjunction of formulae in $Y^2_t$. The point of all this is that axioms of the form $\forall x_1 \ldots \forall x_t (\chi \implies \Phi)$ have a very natural combinatorial interpretation: such an axiom is true in $\mathcal M$ if and only if for every $A \subseteq \vert {\mathcal M} \vert$ with $\vert A \vert = t$,  the isomorphism type of  ${\mathcal M} \vert_A$ is among those specified by the elements of $Y^2_t$ appearing in $\Phi$. In keeping with the usual usage in graph theory, we can view axioms of this kind as saying that certain induced substructures are {\em forbidden}: what we have shown is that for any universal theory $T$ there is a sequence $\langle {\rm Forbid}^T_t : t > 0 \rangle$ such that ${\rm Forbid}^T_t$ is a set of isomorphism types for $\mathcal L$-structures of size $t$, and the models of $T$ are exactly those $\mathcal M$ such that the isomorphism type of ${\mathcal M} \vert_A$ is not in  ${\rm Forbid}^T_t$ for every $t$ and every $A \subseteq \vert {\mathcal M} \vert$ with $\vert A \vert = t$.

\bibliographystyle{plain}

\bibliography{library}

\end{document}